\documentclass{article}
\usepackage[utf8]{inputenc}
\usepackage[a4paper,margin=1in]{geometry}

\usepackage{lineno}
\usepackage{amsmath}
\usepackage{amssymb}
\usepackage{amsthm}
\usepackage{graphicx}
\usepackage{bbm} %to use the font of the characteristic function
\usepackage{subfigure}

\newtheorem{theorem}{Theorem}

\newtheorem{lemma}{Lemma}

\newcommand{\Prob}[1]{\mathbbm{P}\left(#1\right)}
\newcommand{\rbra}[1]{\left( #1 \right)}
\newcommand{\E}[1]{\mathbbm{E}\left(#1\right)}
\newcommand{\Var}[1]{\text{Var}\left(#1\right)}

\def\P{ {\mathbbm{P} }}
\def\S{ {S}}

\title{The variance of the average depth of a pure birth process converges to 7}

\author{Ken R. Duffy \thanks{Hamilton Institute, Maynooth University, Maynooth, Ireland} \and Gianfelice Meli \footnotemark[1] \and Seva Shneer \thanks{School of Mathematical and Computer Sciences, Heriot-Watt University, Edinburgh, UK} }

\date{}

\begin{document}

\maketitle

\begin{abstract}
If trees are constructed from a pure birth process and one defines the
depth of a leaf to be the number of edges to its root, it is known that
the variance in the depth of a randomly selected leaf of a randomly
selected tree grows linearly in time. In this letter, we instead
consider the variance of the average depth of leaves
within each individual tree, establishing that, in contrast, it
converges to a constant, $7$. This result indicates that while the
variance in leaf depths amongst the ensemble of pure birth processes
undergoes large fluctuations, the average depth across individual
trees is much more consistent.
\end{abstract}

%\end{frontmatter}

%\linenumbers

\section{Introduction}

Continuous time branching processes form fundamental building blocks
of many stochastic models (e.g. \cite{Kimmel02}) and much is known
about many statistics associated with them. A pure birth
process~\cite{resnick2013adventures} is the simplest continuous
time branching process. It describes the growth of a directed tree
that starts at time $0$ with a root, which is the first leaf. Each
leaf extends the tree by creating two new leaves after an exponentially
distributed time with mean $1/\lambda$, independently of everything
else. Pure birth processes appear as a fundamental model of study
in a large number of applications from data structures in computer
science to likelihood methods in phylogenetics to the study of
random walkers on random graphs, and are well studied.

Of interest to us here is a measure of tree depth, the distance
from root to leaves. If one conditions on the number of nodes, much
is known. For example, Pittel~\cite{pittel1984growing} linked prior
results regarding binary search
trees~\cite{robson1979height,flajolet1980exploring,devroye1986note} 
to continuous time Markovian branching processes, establishing scaling
properties of the depth of the both the shortest and longest leaf. Further
extensions of those results have since been found
~\cite{pittel1994note,biggins1997note}. Without conditioning on the
number of nodes in the tree, relatively little appears in the
literature. For a pure birth process, it is known that the mean
depth of a randomly chosen leaf in a randomly selected tree grows
as $2\lambda t$ with variance $2\lambda t$~\cite{samuels1971distribution}.
However, for many applications, particularly in the life sciences
e.g \cite{perie14,marchingo2016}, one is interested in the properties
of individual growing trees. Denoting the number of leaves in a
random tree at time $t$ by $Z(t)$ and the sum of their depths by
$G(t)$, with $Z(0)=1$ and $G(0)=0$.  The object of the present study
is the variance across trees of the average depth of the leaves
within them, i.e. $G(t)/Z(t)$, and our main result is as follows.

\begin{theorem}\label{prop:var(G/Z)}
For a pure birth process, we have that
\begin{equation*}
    \lim_{t \to \infty} \Var{\frac{G(t)}{Z(t)}}=7.
\end{equation*}
\end{theorem}

In addition to the results in~\cite{samuels1971distribution}, this
finding is potentially surprising because it is known that the two
processes $\{Z(t)\}$ and $\{G(t)\}$ have different growth rates,
$e^{\lambda t}$ and $te^{\lambda t}$, respectively
~\cite{jagers1969renewal,weber2016inferring}, from which one might
anticipate that the variability of the average depth of a tree
diverges to infinity as $t^2$. Those suppositions are incorrect as
it has recently been established that, for general continuous time
branching processes, $Z(t)$ and $G(t)$ are strongly correlated at
the level of sample paths~\cite{meli2018arXiv180707031M}, and
that, for a pure birth process, $\lim_t Z(t)/(tG(t))=2\lambda$ almost
surely. A visualization of the result in Theorem \ref{prop:var(G/Z)},
obtained by Monte Carlo simulation, is provided in
Fig.~\ref{fig:variance_avg_gen}.  Note that the result does not
depend on $\lambda$, which only influences the speed of convergence.

In order to evaluate $\Var{G(t)/Z(t)}$, we condition the average
generation $G(t)/Z(t)$ on the number of leaves at time $t$, $Z(t)$.
By the Law of Total Variance (e.g.~\cite{blitzstein2014introduction})
\begin{equation}\label{eq:law_total_variance}
    \Var{\frac{G(t)}{Z(t)}}=\E{\Var{\frac{G(t)}{Z(t)}\Big|Z(t)}} + \Var{\E{\frac{G(t)}{Z(t)}\Big|Z(t)}}
\end{equation}
and, in order to study the variance of the average depth of the leaves at time $t$, we study the quantities $\E{G(t)/Z(t)|Z(t)}$ and $\Var{G(t)/Z(t)|Z(t)}$ in Lemmas~\ref{thm:E(Var)} and~\ref{thm:Var(E)}, respectively. Theorem~\ref{prop:var(G/Z)} then follows.

\begin{figure}
    \centering
    \subfigure[\label{fig:variance_avg_gen_b}]{
    \includegraphics[scale=0.36]{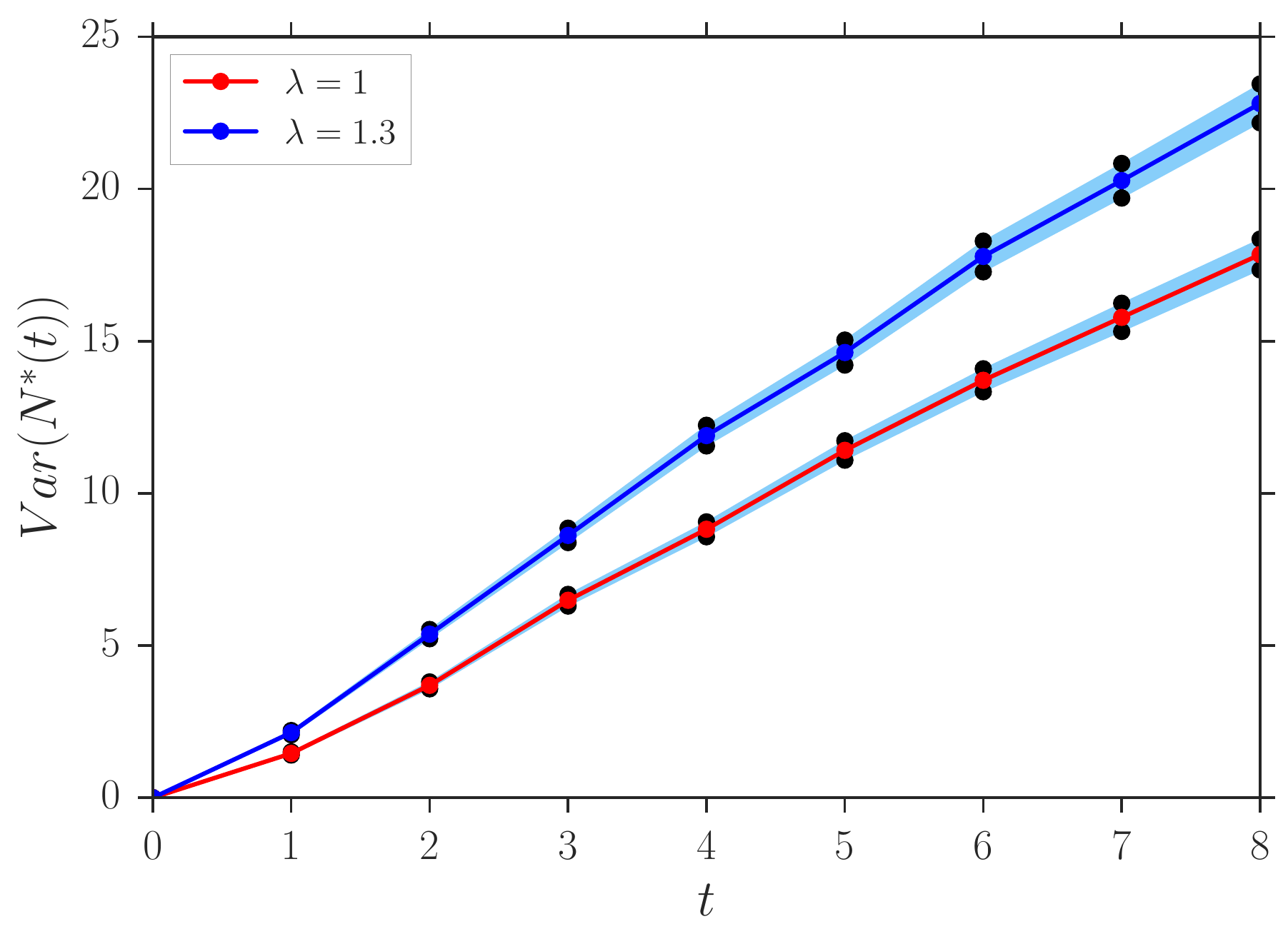}
    }
    \subfigure[\label{fig:variance_avg_gen_a}]{
    \includegraphics[scale=0.36]{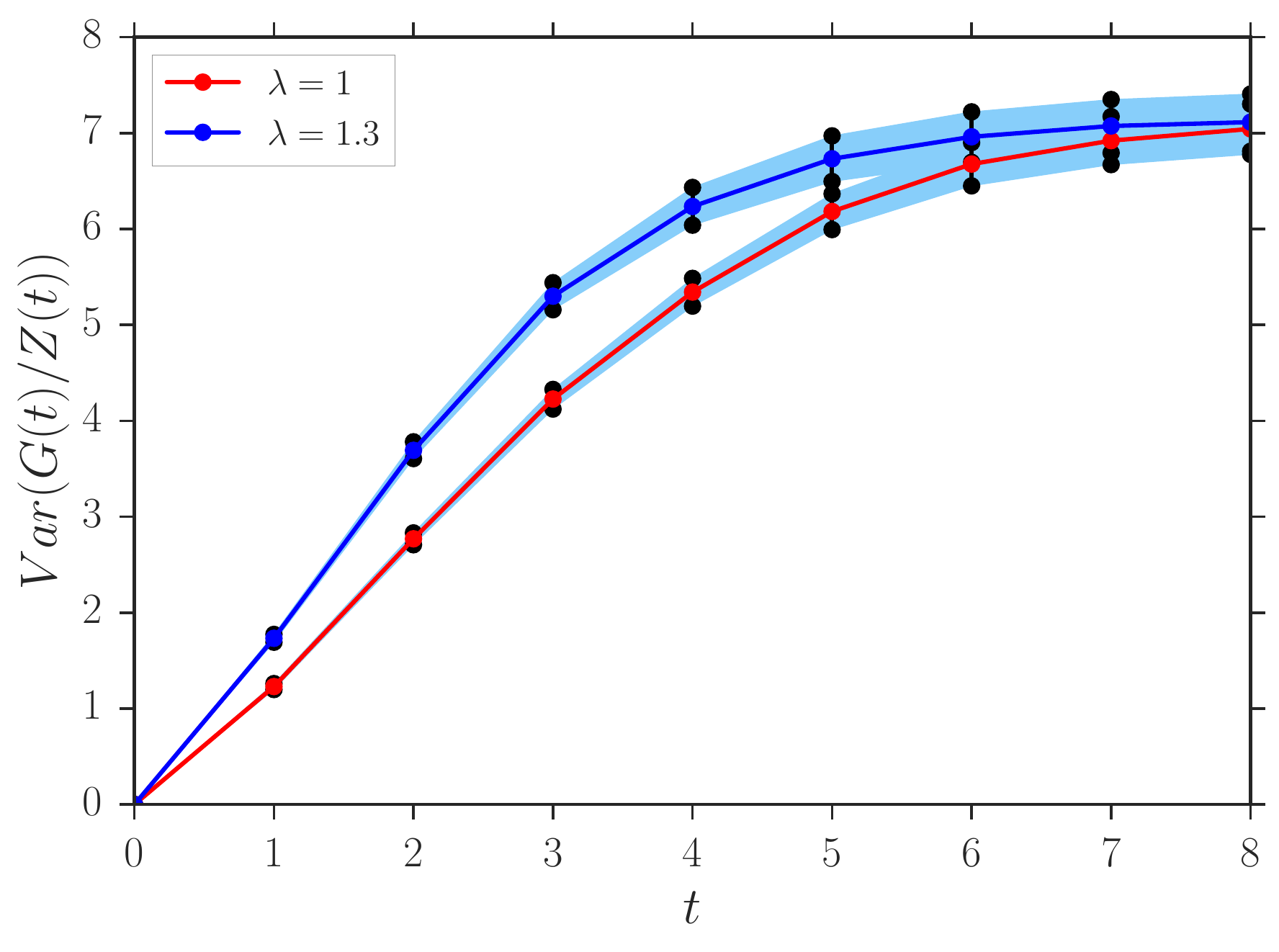}
    }
\caption{$10^4$ Monte Carlo simulations of a pure birth process
were used to determine the variance of the average depth of a leaf
in a random tree, $N^*(t)$, and the variance of the average depth of
a tree, $G(t)/Z(t)$, where $\lambda$ equals 1 (red lines) and 1.3
(blue lines), and the blue shaded region indicates 95\% confidence
intervals based on bootstrap percentiles~\cite[Chapter
13]{efron1994introduction}.
(a) Consistent with~\cite{samuels1971distribution}, $\Var{N^*(t)}\sim
2\lambda t$.
(b) Consistent with Theorem \ref{prop:var(G/Z)},
$\Var{G(t)/Z(t)}\sim 7$ irrespective of $\lambda$.
}\label{fig:variance_avg_gen}
\end{figure}

\section{Results}

Before proceeding with the analysis of the two terms on the RHS
of~\eqref{eq:law_total_variance}, we prove a lemma that will simplify
the proofs of Lemmas~\ref{thm:E(Var)} and~\ref{thm:Var(E)}. For
that, we introduce a new process, $\{\S(t)\}$, denoting the sum of the
squares of the depths of the leaves at time $t$, which appears when
the second moment of $G(t)/Z(t)$ is studied. In the following we
also consider the discrete-time process associated with $\{G(t)\}$
and $\{S(t)\}$, namely $\{G_k\}$ and $\{\S_k\}$, which account for
the sum and the sum of the squares of the depths of the leaves,
respectively, when the number of leaves is $k$.

\begin{lemma}\label{lem:G_k and S_k}
We have that 
\begin{align}
    \E{\frac{G(t)}{Z(t)} \Big| Z(t)=k}&=\frac{\E{G_k}}{k}=2\sum_{i=2}^{k} \frac{1}{i}, \qquad \qquad \frac{\E{\S_k}}{k}= 4 \sum_{i=2}^{k-1} \frac{\E{G_i}}{i(i+1)} + \frac{\E{G_k}}{k}, \label{eq:lemma1_a}\\
    \E{\frac{G(t)^2}{Z(t)^2}\Big|Z(t)=k}&=\frac{\E{G_{k}^2}}{k^2}=\frac{k+1}{k}\rbra{\sum_{i=1}^{k-1} \frac{\E{\S_i}}{i(i+1)(i+2)} + 4\sum_{i=1}^{k-1}\frac{1}{(i+1)(i+2)} + 4\sum_{i=1}^{k-1}\frac{\E{G_i}}{i(i+2)}}.\label{eq:lemma1_b}
\end{align}
\end{lemma}

\begin{proof}
Throughout this proof, we condition on $Z(t)=k$ and denote by
$\Gamma_1, \Gamma_2, \ldots, \Gamma_k$ the depth of the $k$ leaves
present at time $t$, which are not independent.  From
the definitions, we have $G_k:=\sum_{i=1}^k \Gamma_i$ and
$\S_k:=\sum_{i=1}^k \Gamma_i^2$. The idea of the proof is to recover
the formulas given above by finding recurrence equations for $\E{G_k},
\E{S_k}$, and $\E{G_k^2}$.

For $j \in \{1,2,\ldots,k\}$, denote by $I_j$ a random variable
that takes value $1$ if the $j$-th leaf is the first one, among the
$k$ existing, to extend the tree with two new leaves, and $0$
otherwise. The random variables in the set $\{I_j, \Gamma_1, \ldots, \Gamma_k\}$
are independent for $j \in \{1,2,\ldots,k\}$ and, due to the
memoryless property of the exponential distribution, $\Prob{I_j=1}=1/k$
for all $j \in \{1,2,\ldots,k\}$, with $k$ the number of leaves in
the tree. Furthermore, the $I_j$ are not independent of each other
because only one of them can assume value $1$, i.e. $\sum_{j=1}^k
I_j=1$, implying that  $I_j^2=I_j$ and $I_j I_\ell=0$ if $j\neq
\ell$. With that in mind, we establish the following relations

\begin{align}
    G_{k+1}=G_k + &\sum_{j=1}^k I_j \Gamma_j +2, \qquad \qquad S_{k+1}=\S_k + \sum_{j=1}^k I_j \rbra{2(\Gamma_j+1)^2-\Gamma_j^2}, \label{eq:G_k+1_and_S_k+1} \\
    G_{k+1}^2&=G_k^2 + \rbra{\sum_{j=1}^k I_j \Gamma_j}^2 +4 + 4G_k + 4\sum_{j=1}^k I_j \Gamma_j + 2G_k\sum_{j=1}^k I_j \Gamma_j \notag \\
    &=G_k^2 + \sum_{j=1}^k I_j \Gamma_j^2 +4 + 4G_k + 4\sum_{j=1}^k I_j \Gamma_j + 2G_k\sum_{j=1}^k I_j \Gamma_j. \label{eq:G^2_k+1}
\end{align}
From the first equation in~\eqref{eq:G_k+1_and_S_k+1} we obtain
\begin{align*}
    \E{G_{k+1}}&=\E{G_k} + \sum_{j=1}^k\E{I_j \Gamma_j} +2 = \E{G_k} + \sum_{j=1}^k\E{I_j}\E{\Gamma_j} +2 \notag \\ 
    &= \E{G_k} + \frac{1}{k}\E{\sum_{j=1}^k \Gamma_j} +2 = \E{G_k} + \frac{1}{k}\E{G_k} +2 = \frac{k+1}{k}\E{G_k} +2,
\end{align*}
where we have used that $I_j$ and $\Gamma_j$ are independent. This gives the following recurrence relation
$\E{G_{k+1}}(k+1)^{-1}=\E{G_{k}}k^{-1}+ 2(k+1)^{-1}$,
that, solved with initial condition $\E{G_1}=0$, results in the first formula in~\eqref{eq:lemma1_a}.

Similarly, using the second equation in~\eqref{eq:G_k+1_and_S_k+1}, we have that
\begin{align*}
    \E{\S_{k+1}}&= \E{\S_k}+ \sum_{j=1}^k \E{I_j} \E{2(\Gamma_j+1)^2-\Gamma_j^2}
    =\E{\S_k}+\frac{1}{k}\sum_{j=1}^k \E{\Gamma_j^2 +4\Gamma_j +2} \\
    &= \E{\S_k} + \frac{1}{k}\E{\S_k} + \frac{4}{k}\E{G_k} + 2= \frac{k+1}{k}\E{\S_k} + \frac{4}{k}\E{G_k} + 2,
\end{align*}
from which we get the recurrence equation $\E{\S_{k+1}}(k+1)^{-1}
= \E{\S_k}k^{-1} + 4\E{G_k}(k(k+1))^{-1}+2(k+1)^{-1}$. Solving this
recursion with $\E{\S_1}=\E{G_1}=0$, we obtain the second result
in~\eqref{eq:lemma1_a}.

Using~\eqref{eq:G^2_k+1} and the two results just found (i.e. the formulas in~\eqref{eq:lemma1_a}), we can now find an expression for $\E{G_k^2}$. 
\begin{align*}
    \E{G_{k+1}^2}=& \E{G_k^2} + \frac{1}{k} \sum_{j=1}^k \E{\Gamma_j^2} +4 + 4\E{G_k} + \frac{4}{k}\sum_{j=1}^k \E{\Gamma_j} + 2\E{G_k\Big(\sum_{j=1}^k I_j \Gamma_j\Big)} \\
    =& \E{G_k^2}+ \frac{1}{k} \E{\S_k} + 4 + \left(4+\frac{4}{k}\right)\E{G_k} + \frac{2}{k} \E{G_k \sum_{j=1}^k \Gamma_j} \\
    =& \E{G_k^2}+ \frac{\E{\S_k}}{k}  + 4 + \frac{4(k+1)}{k}\E{G_k}  + \frac{2}{k} \E{G_k^2} \\
    =&\frac{k+2}{k} \E{G_k^2}+ \frac{\E{\S_k}}{k}  + 4 + \frac{4(k+1)}{k}\E{G_k}. 
\end{align*}
The equation above can be rewritten as the recurrence equation
\begin{equation*}
    \frac{\E{G_{k+1}^2}}{(k+1)(k+2)}= \frac{\E{G_k^2}}{k(k+1)} + \frac{\E{\S_k}}{k(k+1)(k+2)}  + \frac{4}{(k+1)(k+2)} + \frac{4\E{G_k}}{k(k+2)},
\end{equation*}
that, when solved with initial condition $\E{G_1}=\E{G_1^2}=\E{\S_1}=0$, gives~\eqref{eq:lemma1_b}.
\end{proof}

We now use Lemma~\ref{lem:G_k and S_k} to study the limit behaviour of the first term on the RHS of~\eqref{eq:law_total_variance}.

\begin{lemma}\label{thm:E(Var)}
For a pure birth process, we have that 
\begin{equation*}
    \lim_{t \to \infty} \E{\Var{\frac{G(t)}{Z(t)} \Big| Z(t) } }= 7-\frac{2}{3}\pi^2.
\end{equation*}
\end{lemma}
\begin{proof}

Given that $\lim_{t \to \infty} Z(t)=\infty$ a.s.~\cite[Chapter 5]{harris1963theory},  for every fixed  $k \in \mathbbm{N}$ we have that $\lim_{t \to \infty} \Prob{Z(t)=k}=0$. This implies that 
\begin{equation*}
    \lim_{t \to \infty}\E{\Var{\frac{G(t)}{Z(t)}\Big|Z(t)}}= \lim_{t \to \infty} \sum_{k=1}^\infty \Var{\frac{G_k}{k}\Big|Z(t)=k} \Prob{Z(t)=k} = \lim_{k \to \infty}  \Var{\frac{G_k}{k}}. 
\end{equation*}

Using Lemma~\ref{lem:G_k and S_k}, we can now compute this variance:
\begin{align*}
    \Var{\frac{G_k}{k}}&=  \frac{\E{G_k^2}}{k^2}-\frac{\E{G_k}^2}{k^2}\\
    &= \frac{k+1}{k}\bigg[\sum_{i=1}^{k-1}\frac{\E{\S_i}}{i(i+1)(i+2)} + 4\sum_{i=1}^{k-1}\frac{1}{(i+1)(i+2)} + 4\sum_{i=1}^{k-1}\frac{\E{G_i}}{i(i+2)}\bigg] - \frac{\E{G_k}^2}{k^2}\\
    &= \frac{k+1}{k}\bigg[\sum_{i=1}^{k-1}\frac{\E{\S_i}}{i(i+1)(i+2)} + 4\sum_{i=1}^{k-1}\frac{1}{(i+1)(i+2)} + 4\sum_{i=1}^{k}\frac{\E{G_i}}{i^2}\\
    & \quad+ 4\sum_{i=1}^{k-1}\Big( \frac{\E{G_i}}{i(i+2)} - \frac{\E{G_i}}{i^2} \Big) -4 \frac{\E{G_k}}{k^2}\bigg] - \frac{\E{G_k}^2}{k^2}\\
    &= \frac{k+1}{k}\bigg[\sum_{i=1}^{k-1}\frac{\E{\S_i}}{i(i+1)(i+2)} + 4\sum_{i=1}^{k-1}\frac{1}{(i+1)(i+2)} + 4 \sum_{i=2}^{k} \frac{1}{i^2} + \frac{\E{G_k}^2}{k^2}  \\
    &\quad - 8 \sum_{i=1}^{k-1}\frac{\E{G_i}}{i^2(i+2)}  -4 \frac{\E{G_k}}{k^2}\bigg] - \frac{\E{G_k}^2}{k^2}\\
    &= \frac{k+1}{k}\bigg[\sum_{i=1}^{k-1}\frac{\E{\S_i}}{i(i+1)(i+2)} + 4\sum_{i=1}^{k-1}\frac{1}{(i+1)(i+2)} + 4 \sum_{i=2}^{k} \frac{1}{i^2}  \\
    &\quad - 8 \sum_{i=1}^{k-1}\frac{\E{G_i}}{i^2(i+2)}  -4 \frac{\E{G_k}}{k^2}\bigg] + \frac{\E{G_k}^2}{k^3},
\end{align*}
where in the third equality we have added and subtracted the quantity
\begin{align*}
    4\sum_{i=1}^{k}\frac{\E{G_i}}{i^2} =& 8 \sum_{i=2}^{k} \frac{1}{i} \sum_{j=2}^{i} \frac{1}{j} = 8 \sum_{i=2}^{k} \frac{1}{i^2} + 8 \sum_{i=2}^{k} \sum_{j=2}^{i-1}\frac{1}{ij} \\
    =&8 \sum_{i=2}^{k} \frac{1}{i^2} + 4 \rbra{ \rbra{\sum_{i=2}^{k}\frac{1}{i}}^2 - \sum_{i=2}^{k} \frac{1}{i^2}} 
    = 4 \sum_{i=2}^{k} \frac{1}{i^2} + \frac{\E{G_k}^2}{k^2}.
\end{align*}
Taking the limit as $k \to \infty$, we have that
\begin{align}\label{eq:lim_var}
    \lim_{k \to \infty} \Var{\frac{G_k}{k}}=& \sum_{i=1}^{\infty}\frac{\E{\S_i}}{i(i+1)(i+2)} + 4\sum_{i=1}^{\infty}\frac{1}{(i+1)(i+2)} + 4\sum_{i=2}^{\infty}\frac{1}{i^2} - 8 \sum_{i=1}^{\infty}\frac{\E{G_i}}{i^2(i+2)} \notag\\
    =& \sum_{i=1}^{\infty}\frac{\E{\S_i}}{i(i+1)(i+2)} +2 +4 \Big(\frac{\pi^2}{6} -1\Big) - 8 \sum_{i=1}^{\infty}\frac{\E{G_i}}{i^2(i+2)}.
\end{align}
Using Lemma~\ref{lem:G_k and S_k}, the first term on the RHS of~\eqref{eq:lim_var} becomes
\begin{multline}\label{eq:first term}
    \sum_{i=1}^{\infty}\frac{1}{(i+1)(i+2)}\left(4 \sum_{k=2}^{i-1} \frac{1}{k+1} \frac{\E{G_k}}{k} + \frac{\E{G_i}}{i} \right) \\
    = 8 \sum_{i=1}^{\infty}\frac{1}{(i+1)(i+2)} \sum_{k=2}^{i-1} \frac{1}{k+1} \sum_{j=2}^k \frac{1}{j}
    + 2\sum_{i=1}^{\infty}\frac{1}{(i+1)(i+2)}\sum_{k=2}^i \frac{1}{k}.
\end{multline}
The first term on the RHS of~\eqref{eq:first term} is given by
\begin{align*}
    8 \sum_{j=2}^{\infty} \frac{1}{j} \sum_{k=j}^{\infty} \frac{1}{k+1} \sum_{i=k+1}^{\infty} \frac{1}{(i+1)(i+2)} = 8 \sum_{j=2}^{\infty} \frac{1}{j} \sum_{k=j}^{\infty} \frac{1}{k+1}\frac{1}{k+2}=8 \sum_{j=2}^{\infty} \frac{1}{j}\frac{1}{j+1}=4,
\end{align*}
whereas the second one is given by
\begin{equation*}
    2\sum_{k=2}^{\infty} \frac{1}{k} \sum_{i=k}^{\infty}\frac{1}{(i+1)(i+2)}=2\sum_{k=2}^{\infty} \frac{1}{k} \frac{1}{k+1}=1.
\end{equation*}
So, the first sum in the RHS of~\eqref{eq:lim_var} is equal to $4+1=5$.
For the last sum in the RHS of~\eqref{eq:lim_var}, we have
\begin{align*}
    - 8 \sum_{i=1}^{\infty}\frac{1}{i(i+2)} \frac{\E{G_i}}{i}&= - 16 \sum_{i=1}^{\infty}\frac{1}{i(i+2)} \sum_{j=2}^i \frac{1}{j}=- 16 \sum_{j=2}^{\infty} \frac{1}{j} \sum_{i=j}^{\infty}\frac{1}{i(i+2)} \\
    &= - 16 \sum_{j=2}^{\infty} \frac{1}{j} \frac{1+2j}{2j(j+1)}= -\frac{4}{3}(\pi^2-3).
\end{align*}
Joining all these results, we obtain
\begin{equation*}
    \lim_{k \to \infty} \Var{\frac{G_k}{k}}=5 + 2 +4 \Big(\frac{\pi^2}{6} -1\Big) -\frac{4}{3}(\pi^2-3)=7-\frac{2}{3}\pi^2.
\end{equation*}
\end{proof}
Lemma~\ref{lem:G_k and S_k} allows us to also understand the behaviour of the conditional variance of the expected average depth of the leaves given their number.

\begin{lemma}\label{thm:Var(E)}
For a pure birth process, we have that  
\begin{equation*}
    \lim_{t \to \infty} \Var{\E{\frac{G(t)}{Z(t)} | Z(t) } }= \frac{2}{3}\pi^2 \approx 6.58.
\end{equation*}
\end{lemma}
\begin{proof}
From Lemma~\ref{lem:G_k and S_k} we know that
\begin{align} \label{eq_variance_1}
\Var{\E{\frac{G(t)}{Z(t)} | Z(t) } }&=\Var{2\sum_{i=2}^{Z(t)} \frac{1}{i} } = 4\Var{\sum_{i=1}^{Z(t)} \frac{1}{i}} = 4\rbra{\E{\left(\sum_{i=1}^{Z(t)}\frac{1}{i}\right)^2} - \left(\E{\sum_{i=1}^{Z(t)} \frac{1}{i} } \right)^2},
\end{align}
where, in the second inequality, we have used the fact that the
variance of a process doesn't change when a constant is added.
Given that $\{Z(t)\}$ is a pure birth process, the distribution 
of $Z(t)$ is given by (e.g.~\cite[pg. 430]{resnick2013adventures})
\begin{equation*}
    \Prob{Z(t) = k} = e^{-\lambda t} (1-e^{-\lambda t})^{k-1}, \quad k=1,2,\ldots
\end{equation*}
where $1/\lambda$ is the expected time before a leaf generates two new leaves,
which allows us to evaluate the second term in \eqref{eq_variance_1} exactly: 
\begin{align*}
\E{\sum_{i=1}^{Z(t)} \frac{1}{i} } &=\sum_{k=1}^\infty \P(Z(t)=k) \sum_{i=1}^k \frac{1}{i} = \frac{e^{-\lambda t}}{1-e^{-\lambda t}} \sum_{k=1}^\infty (1-e^{- \lambda t})^{k} \sum_{i=1}^k \frac{1}{i} \\ 
&=\frac{e^{-\lambda t}}{1-e^{-\lambda t}} \sum_{i=1}^\infty \frac{1}{i} \sum_{k=i}^\infty (1-e^{-\lambda t})^{k} = \frac{1}{(1-e^{-\lambda t})} \sum_{i=1}^\infty \frac{1}{i} (1-e^{-\lambda t})^{i}.
\end{align*}
Let $f(t) := \sum_{i=1}^\infty  (1-e^{-\lambda t})^{i}/i$. Then
\begin{equation*}
    f'(t) = \lambda e^{-\lambda t} \sum_{i=1}^\infty \frac{1}{i} i (1-e^{-\lambda t})^{i-1} = \lambda e^{-\lambda t} \sum_{i=1}^\infty (1-e^{-\lambda t})^{i-1} = \lambda,
\end{equation*}
and, given $f(0)=0$, we have that $f(t)=\lambda t$. This implies that 
\begin{equation}\label{eq:mean_g_z}
    \E{\sum_{i=1}^{Z(t)} \frac{1}{i} } = \frac{\lambda t}{(1-e^{-\lambda t})} = \lambda t + o(1),
\end{equation}
and the second term in the brackets on the RHS of \eqref{eq_variance_1} is therefore $(\lambda^2 t^2)/(1-e^{-\lambda t})^2$.

Consider the first term on the RHS of \eqref{eq_variance_1}.
\begin{align} \label{eq_variance_2}
\E{\left(\sum_{i=1}^{Z(t)} \frac{1}{i}\right)^2} &= \frac{e^{-\lambda t}}{1-e^{-\lambda t}} \sum_{i=1}^\infty \left(\sum_{k=1}^i \frac{1}{k}\right)^2 (1-e^{-\lambda t})^{i} \notag \\
&= \frac{e^{-\lambda t}}{1-e^{-\lambda t}} \left( \sum_{i=1}^\infty \sum_{k=1}^i \frac{1}{k^2} (1-e^{-\lambda t})^{i} + 2 \sum_{i=1}^\infty \sum_{k=1}^i \sum_{j=1}^{k-1} \frac{1}{k} \frac{1}{j} (1-e^{-\lambda t})^{i} \right).
\end{align}
The first term in the brackets on the RHS of~\eqref{eq_variance_2} is given by
\begin{equation*}
    \sum_{i=1}^\infty \sum_{k=1}^i \frac{1}{k^2} (1-e^{-\lambda t})^{i} = \sum_{k=1}^\infty \sum_{i=k}^\infty \frac{1}{k^2} (1-e^{-\lambda t})^{i} = e^{\lambda t} \sum_{k=1}^\infty \frac{1}{k^2} (1-e^{-\lambda t})^{k}.
\end{equation*}
For the second term, we have that
\begin{equation*}
  2\sum_{i=1}^\infty \sum_{k=1}^i \sum_{j=1}^{k-1} \frac{1}{k} \frac{1}{j} (1-e^{-\lambda t})^{i} = 2\sum_{k=1}^\infty \sum_{j=1}^{k-1} \sum_{k=i}^{\infty} \frac{1}{k} \frac{1}{j} (1-e^{-\lambda t})^{i} = 2e^{\lambda t} \sum_{k=1}^\infty \sum_{j=1}^{k-1} \frac{1}{k} \frac{1}{j} (1-e^{-\lambda t})^{k}.  
\end{equation*}
Denoting with $g(t) := 2\sum_{k=1}^\infty \sum_{j=1}^{k-1} (1-e^{-\lambda t})^{k}/(kj)$ and noticing that $g(0)=0$ and
\begin{align*}
    g'(t) &= \frac{2 \lambda e^{-\lambda t}}{1-e^{-\lambda t}} \sum_{k=1}^\infty \sum_{j=1}^{k-1} \frac{1}{j} (1-e^{-\lambda t})^{k} = \frac{2 \lambda e^{-\lambda t}}{1-e^{-\lambda t}} \sum_{j=1}^{\infty} \sum_{k=j+1}^\infty \frac{1}{j} (1-e^{-\lambda t})^{k} \\ 
    &= \frac{2 \lambda }{1-e^{-\lambda t}} \sum_{j=1}^{\infty} \frac{1}{j} (1-e^{-\lambda t})^{j+1} = 2\lambda f(t) = 2\lambda^2 t,
\end{align*}
we obtain that $g_2(t) = \lambda^2 t^2$, and the second term on the RHS of \eqref{eq_variance_2}  is thus $(\lambda^2 t^2)/(1-e^{-\lambda t})$.

So, joining all the results, we have that 
\begin{align*}
    \lim_{t \to \infty} \Var{\E{\frac{G(t)}{Z(t)} | Z(t)  } }=  \lim_{t \to \infty} 4\rbra{  \sum_{k=1}^\infty \frac{(1-e^{-\lambda t})^{k-1}}{k^2}  + \frac{\lambda^2 t^2}{1-e^{-\lambda t}} - \frac{\lambda^2 t^2}{(1-e^{-\lambda t})^2}} = 4 \sum_{k=1}^\infty \frac{1}{k^2}  = \frac{2}{3} \pi^2
\end{align*}

\end{proof}

Theorem~\ref{prop:var(G/Z)} follows from equation \eqref{eq:law_total_variance} 
using the results in Lemmas \ref{thm:E(Var)} and \ref{thm:Var(E)}.

{\bf Acknowledgments:} The authors thank Tom S. Weber (WEHI) for contributing to the conjecture of Theorem~\ref{prop:var(G/Z)}.
Part of this work was supported by Science Foundation Ireland grant 12 IP 1263.

\end{document}